\documentclass{amsart}
\usepackage{color,xspace}
\usepackage[normalem]{ulem}
\usepackage{mathrsfs}
\usepackage{stmaryrd}
\usepackage{graphicx}

\newcommand\RE{\mathbb{R}}
\newcommand\CO{\mathbb{C}}

\newcommand\sfzero{\mathsf{0}}
\newcommand\sfA{\mathsf{A}}
\newcommand\sfB{\mathsf{B}}
\newcommand\sfC{\mathsf{C}}
\newcommand\sfD{\mathsf{D}}
\newcommand\sfE{\mathsf{E}}
\newcommand\sfF{\mathsf{F}}
\newcommand\sfM{\mathsf{M}}

\newcommand\sfx{\mathsf{x}}
\newcommand\sfy{\mathsf{y}}
\newcommand\sfz{\mathsf{z}}
\newcommand\Tf{T_{F1}}
\newcommand\Tfh{T_{F1,h}}
\newcommand\Tl{T_{LL*}}
\newcommand\Tlh{T_{LL*,h}}

\newcommand\bfsigma{\boldsymbol{\sigma}}
\newcommand\bftau{\boldsymbol{\tau}}
\newcommand\bfchi{\boldsymbol{\chi}}
\newcommand\bfxi{\boldsymbol{\xi}}
\newcommand\bfH{\mathbf{H}}
\newcommand\bft{\mathbf{t}}
\newcommand\bfn{\mathbf{n}}
\newcommand\grad{\operatorname\nabla}
\renewcommand\div{\operatorname{\mathrm{div}}}
\newcommand\curl{\operatorname{\mathbf{curl}}}
\newcommand\Hdiv{\mathbf{H}(\div;\Omega)}
\newcommand\Hcurl{\mathbf{H}(\curl;\Omega)}
\newcommand\Hocurl{\mathbf{H}_0(\curl;\Omega)}
\newcommand\llstar{LL$^*$\xspace}
\newcommand\Span{\operatorname{\mathrm{span}}}
\newcommand\rank{\operatorname{\mathrm{rank}}}

\newtheorem{theorem}{Theorem}
\newtheorem{lemma}[theorem]{Lemma}
\newtheorem{proposition}[theorem]{Proposition}

\theoremstyle{remark}
\newtheorem{remark}{Remark}

\begin{document}

\title{First order least-squares formulations for eigenvalue problems}

\author{Fleurianne Bertrand}
\address{Humboldt-Universit\"at zu Berlin, Germany and King Abdullah
University of Science and Technology, Saudi Arabia}
\curraddr{}
\email{}
\thanks{}

\author{Daniele Boffi}
\address{King Abdullah University of Science and Technology, Saudi Arabia, and
University of Pavia, Italy}
\curraddr{}
\email{}
\thanks{}

\date{}

\begin{abstract}

In this paper we discuss spectral properties of operators associated with the
least-squares finite element approximation of elliptic partial differential
equations. The convergence of the discrete eigenvalues and eigenfunctions
towards the corresponding continuous eigenmodes is studied and analyzed with
the help of appropriate $L^2$ error estimates. A priori and a posteriori
estimates are proved.

\end{abstract}

\maketitle

\section{Introduction}

Least-squares finite element formulations have been successfully used for the
approximation of several problems described in terms of partial differential
equations.

In particular, we are considering formulations that approximate
simultaneously scalar (potential) and vector (flux) variables in the spirit of
first-order system least squares~\cite{bochev}. While least-squares schemes
posses an inherent error control and are particularly suited for problems
involving coupling conditions, other approaches involving mixed or hybrid
schemes~\cite{bbf} enjoy good conservation properties. The closedness property
from~\cite{brandts} show how sometimes results from one approach can be
transferred to the other one.

Only few papers deal with eigenvalue problems associated with least-squares
formulations. In~\cite{BrambleOsborn} the authors apply their theory to a
second order least-squares formulation of a Dirichlet eigenvalue problem.
In~\cite{pasciak} a first order least-squares formulation is introduced for
the approximation of the eigenvalues of Maxwell's equations.

In this paper we aim at investigating the least-squares finite element
approximation of the eigensolutions of operators associated with second order
elliptic equations. Even if the proposed method may be not competitive with
other solution techniques, the presented analysis sheds some light on
fundamental properties of least-squares formulations, in particular in
connection with the simulation of evolution problems.

We start with presenting several least-squares formulations for the
approximation of the eigensolutions of the Diriclet Laplace problem. For each
formulation we characterize the eigenmodes obtained after finite element
discretization and we describe the structure of the underlined algebraic
systems.

We then discuss the convergence of the discrete solutions towards the
continuous eigenmodes. We use the standard theory of the approximation of
compact operators (see~\cite{BaOs,acta} and the references therein); it can be
easily seen that standard energy estimates (in the graph norm) are not enough
to guarantee the uniform convergence of the discrete solution operator
sequence to the continuous solution operator. This is a consequence of the
known lack of compactness of the solution operator in the energy norm; for
this reason, we consider the solution operator in $L^2(\Omega)$ and we discuss
various $L^2(\Omega)$ error estimates. It turns out that in the case of $\div$
formulations for FOSLS (First order system least-squares) and \llstar
formulations, if the flux variable is approximated with Raviart--Thomas spaces
(or, in general, with other mixed spaces~\cite{bbf}) then the presented
approximations are optimally convergent. On the other hand, the corresponding
$\div$--$\curl$ formulations suffer, as expected, from serious issues when
applied to singular solutions such as those occurring when the computational
domain presents reentrant corners; in this case continuous finite elements
cannot correctly approximate the flux which is not $\bfH^1(\Omega)$ regular
and the corresponding eigenvalues converge to a wrong solution.

A priori and a posteriori error estimates are presented and rigorously proved
for the proposed formulations.

Several numerical tests conclude the paper, confirming our results and
investigating situations not covered by the theory.

\section{The Laplace eigenvalue problem}

The problem we are considering is to find $\lambda\in\RE$ and $u$ non
vanishing such that
\[
\left\{
\aligned
&-\Delta u=\lambda u&&\text{in }\Omega\\
&u=0&&\text{on }\partial\Omega
\endaligned
\right.
\]
Our problem can be written in the following standard first order formulation:
find $\lambda\in\RE$ and $u$ non vanishing such that for some $\bfsigma$
\[
\left\{
\aligned
&\bfsigma-\grad u=0&&\text{in }\Omega\\
&\div\bfsigma=-\lambda u&&\text{in }\Omega\\
&u=0&&\text{on }\partial\Omega
\endaligned
\right.
\]

More general symmetric elliptic problems in divergence form could be
considered, as well as different homogeneous boundary conditions. Since all
our analysis applies with standard modifications to more general situations,
we describe our theory in the simplest possible setting.

\subsection{FOSLS formulation}

The simplest least squares formulation for the source problem is given by the
minimization of the following functional~\cite{fosls}:
\[
\mathcal{F}(\bftau,v)=\|\bftau-\grad v\|^2+\|\div\bftau+f\|^2
\]
If the $L^2(\Omega)$ norm is considered, this leads to the following
variational formulation: find $\bfsigma\in\Hdiv$ and $u\in H^1_0(\Omega)$ such
that
\begin{equation}
\left\{
\aligned
&(\bfsigma,\bftau)+(\div\bfsigma,\div\bftau)-(\grad
u,\bftau)=-(f,\div\bftau)&&\forall\bftau\in\Hdiv\\
&-(\bfsigma,\grad v)+(\grad u,\grad v)=0&&\forall v\in H^1_0(\Omega)
\endaligned
\right.
\label{eq:F1source}
\end{equation}
This formulation can be used in a naturally way to consider the following
eigenvalue problem: find $\lambda\in\CO$ and $u\in H^1_0(\Omega)$ with $u\ne0$
such that for some $\bfsigma\in\Hdiv$ it holds
\begin{equation}
\left\{
\aligned
&(\bfsigma,\bftau)+(\div\bfsigma,\div\bftau)-(\grad
u,\bftau)=-\lambda(u,\div\bftau)&&\forall\bftau\in\Hdiv\\
&-(\bfsigma,\grad v)+(\grad u,\grad v)=0&&\forall v\in H^1_0(\Omega)
\endaligned
\right.
\label{eq:F1}
\tag{F1}
\end{equation}

Even if the formulation is not symmetric, it can be easily shown that the
eigenvalues are real. We state this result in the next proposition since its
proof might have interesting consequences for the numerical approximation of
our problem.

\begin{proposition}

Problem~\eqref{eq:F1} admits a sequence of positive eigenvalues
\[
0<\lambda_1<\lambda_2\le\lambda_3\le\dots
\]
diverging to $+\infty$. The corresponding eigenspaces span the space
$H^1_0(\Omega)$.

\end{proposition}

\begin{proof}

The result follows by the simple observation that the solution operator
associated with problem~\eqref{eq:F1source} is exactly the same as for the
standard Laplace equation. We would like however to show explicitly that the
eigenvalues of~\eqref{eq:F1} are real since this has interesting implications
for the finite element discretization.

The (non symmetric) operator form of problem~\eqref{eq:F1}, with natural
notation, is given by
\begin{equation}
\left(
\begin{matrix}
\sfA&\sfB^\top\\
\sfB&\sfC
\end{matrix}
\right)
\left(
\begin{matrix}
\sfx\\\sfy
\end{matrix}
\right)=
\lambda\left(
\begin{matrix}
\sfzero&\sfD\\
\sfzero&\sfzero
\end{matrix}
\right)
\left(
\begin{matrix}
\sfx\\\sfy
\end{matrix}
\right)
\label{eq:F1matrix}
\end{equation}
After integration by parts, thanks to the boundary conditions we have
$\sfD=-\sfB^\top$ so that~\eqref{eq:F1matrix} can be reduced to the following
equivalent symmetric Schur complement formulation
\begin{equation}
\sfA\sfx=(\lambda+1)\sfB^\top\sfC^{-1}\sfB\sfx,
\label{eq:matr1}
\end{equation}
where we have used the equality $\sfy=-\sfC^{-1}\sfB\sfx$.

Another possible way of observing that~\eqref{eq:F1matrix} corresponds to a
symmetric problem is to rearrange its terms as follows
\[
\left(
\begin{matrix}
\sfA&\sfzero\\
(\lambda+1)\sfB&(\lambda+1)\sfC
\end{matrix}
\right)
\left(
\begin{matrix}
\sfx\\\sfy
\end{matrix}
\right)=
(\lambda+1)\left(
\begin{matrix}
\sfzero&-\sfB^\top\\
\sfzero&\sfzero
\end{matrix}
\right)
\left(
\begin{matrix}
\sfx\\\sfy
\end{matrix}
\right)
\]
obtaining finally
\[
\left(
\begin{matrix}
\sfA&\sfzero\\
\sfzero&\sfzero
\end{matrix}
\right)
\left(
\begin{matrix}
\sfx\\\sfy
\end{matrix}
\right)=
(\lambda+1)\left(
\begin{matrix}
\sfzero&-\sfB^\top\\
-\sfB&-\sfC
\end{matrix}
\right)
\left(
\begin{matrix}
\sfx\\\sfy
\end{matrix}
\right)
\]

\end{proof}

\begin{remark}
One might think that problem~\eqref{eq:F1} (see in particular
formulation~\eqref{eq:F1matrix}) gives a number of infinite eigenvalues;
however, in our formulation of problem~\eqref{eq:F1} the eigenfunctions we are
looking for correspond to the component $u$ of the solution only. We will go
back to this remark later when the approximation of~\eqref{eq:F1} is
considered.
\end{remark}

\subsection{The transpose FOSLS formulation}

Since our problem is self-adjoint, another possibility is to consider the
transpose of~\eqref{eq:F1}: find $\lambda\in\RE$ and $\bfsigma\in\Hdiv$ with
$\bfsigma\ne\mathbf{0}$ such that for some $u\in H^1_0(\Omega)$ it holds
\begin{equation}
\left\{
\aligned
&(\bfsigma,\bftau)+(\div\bfsigma,\div\bftau)-(\grad
u,\bftau)=0&&\forall\bftau\in\Hdiv\\
&-(\bfsigma,\grad v)+(\grad u,\grad v)=-\lambda(\div\bfsigma,v)&&\forall v\in
H^1_0(\Omega)
\endaligned
\right.
\label{eq:F1*}
\tag{F1*}
\end{equation}
This leads to the following operator form
\begin{equation}
\left(
\begin{matrix}
\sfA&\sfB^\top\\
\sfB&\sfC
\end{matrix}
\right)
\left(
\begin{matrix}
\sfx\\\sfy
\end{matrix}
\right)=
\lambda\left(
\begin{matrix}
\sfzero&\sfzero\\
\sfD^\top&\sfzero
\end{matrix}
\right)
\left(
\begin{matrix}
\sfx\\\sfy
\end{matrix}
\right)
\label{eq:F1*matrix}
\end{equation}
and to the corresponding \emph{symmetric} Schur complement form
\begin{equation}
\sfC\sfy=(\lambda+1)\sfB\sfA^{-1}\sfB^\top\sfy
\label{eq:matr2}
\end{equation}

\begin{proposition}
Problems~\eqref{eq:matr1} and~\eqref{eq:matr2} (and hence
formulations~\eqref{eq:F1} and~\eqref{eq:F1*}) are equivalent.
\end{proposition}

\begin{proof}
The equivalence can be seen, for instance, by solving the matrix
problem~\eqref{eq:F1matrix} for $\sfx$, thus obtaining
$\sfx=\sfA^{-1}(\lambda\sfD\sfy-\sfB^\top\sfy)$ which gives
$\sfC\sfy=(\lambda+1)\sfB\sfA^{-1}\sfB^\top\sfy$, that is~\eqref{eq:matr2}.
\end{proof}

\subsection{The \llstar formulation}

Another popular choice for the approximation of the problem under
consideration is the so called \llstar formulation~\cite{llstar}. One of the
reasons for its introduction is the possibility to deal with less regular
right hand sides; moreover, it gives rise to an intrinsically symmetric
formulation, which makes it appealing for the application to eigenvalue
problems. In the case of the source problem it reads: find $\bfchi\in\Hdiv$
and $p\in H^1_0(\Omega)$ such that
\begin{equation}
\left\{
\aligned
&(\bfchi,\bfxi)+(\div\bfchi,\div\bfxi)-(\grad
p,\bfxi)=0&&\forall\bfxi\in\Hdiv\\
&-(\bfchi,\grad q)+(\grad p,\grad q)=(f,q)&&\forall q\in H^1_0(\Omega)
\endaligned
\right.
\label{eq:LL*source}
\end{equation}
It turns out that this formulation is related to our original Laplace problem
by the following relation:
\begin{equation}
\aligned
&-\Delta u=f\\
&-\Delta p = f-u\\
&\bfchi=\grad(p-u)\\
&\div\bfchi=u
\endaligned
\label{eq:relations}
\end{equation}

The eigenvalue problem associated with~\eqref{eq:LL*source} is: find
$\mu\in\RE$ and $p\in H^1_0(\Omega)$, with $p\ne0$, such that for some
$\bfchi\in\Hdiv$ it holds
\begin{equation}
\left\{
\aligned
&(\bfchi,\bfxi)+(\div\bfchi,\div\bfxi)-(\grad
p,\bfxi)=0&&\forall\bfxi\in\Hdiv\\
&-(\bfchi,\grad q)+(\grad p,\grad q)=\mu(p,q)&&\forall q\in H^1_0(\Omega)
\endaligned
\right.
\label{eq:LL*}
\tag{LL*}
\end{equation}
As already anticipated, this problem is symmetric and it can be written in the
following form in terms of the underlined operators:
\[
\left(
\begin{matrix}
\sfA&\sfB^\top\\
\sfB&\sfC
\end{matrix}
\right)
\left(
\begin{matrix}
\sfx\\
\sfy
\end{matrix}
\right)
=\mu
\left(
\begin{matrix}
\sfzero&\sfzero\\
\sfzero&\sfM
\end{matrix}
\right)
\left(
\begin{matrix}
\sfx\\
\sfy
\end{matrix}
\right)
\]

By using the links between the \llstar formulation and the original problem,
as stated in~\eqref{eq:relations}, we can see how to relate the eigenvalues
of~\eqref{eq:LL*} to the ones of the problem we are interested in.

\begin{proposition}
The eigenvalues $\mu$ of~\eqref{eq:LL*} are in one-to-one correspondence with
the eigenvalues $\lambda$ of the Laplace eigenproblem using the relation
\[
\lambda=\frac{\mu+\sqrt{\mu^2+4}}{2}
\]
Moreover, the eigenfunctions $u$ of the Laplace eigenproblem are given by
$\div\bfchi$ and their gradients $\grad u$ are equal to $\grad p-\bfchi$.
\label{pr:ll*}
\end{proposition}

\subsection{Enriching the formulations with $\curl\bfsigma$}
\label{ss:curl}

Since $\bfsigma$ is a gradient, it satisfies $\curl\bfsigma=0$; a commonly
used modification of the FOSLS methods consists in using a least-squares
functional that contains the term $\curl\bfsigma$, that is,
\[
\mathcal{F}(\bftau,v)=\|\bftau-\grad
v\|^2+\|\curl\bfsigma\|^2+\|\div\bftau+f\|^2
\]

With natural modifications the two corresponding formulations read: find
$\lambda\in\RE$ and $u\in H^1_0(\Omega)$ with $u\ne0$ such that for some
$\bfsigma\in\Hdiv\cap\Hcurl$ it holds
\begin{equation}
\left\{
\aligned
&(\bfsigma,\bftau)+(\div\bfsigma,\div\bftau)+(\curl\bfsigma,\curl\bftau)-(\grad
u,\bftau)=-\lambda(u,\div\bftau)\\
&\hspace{6cm}\forall\bftau\in\Hdiv\cap\Hcurl\\
&-(\bfsigma,\grad v)+(\grad u,\grad v)=0
\hspace{1.8cm}\forall v\in H^1_0(\Omega)
\endaligned
\right.
\label{eq:F1curl}
\tag{F1curl}
\end{equation}
and find $\lambda\in\RE$ and $\bfsigma\in\Hdiv\cap\Hcurl$ with
$\bfsigma\ne\mathbf{0}$ such that for some $u\in H^1_0(\Omega)$ it holds
\begin{equation}
\left\{
\aligned
&(\bfsigma,\bftau)+(\div\bfsigma,\div\bftau)+(\curl\bfsigma,\curl\bftau)-(\grad
u,\bftau)=0\\
&\hspace{6.9cm}\forall\bftau\in\Hdiv\cap\Hcurl\\
&-(\bfsigma,\grad v)+(\grad u,\grad v)=-\lambda(\div\bfsigma,v)
\hspace{1.0cm}\forall v\in H^1_0(\Omega)
\endaligned
\right.
\label{eq:F1*curl}
\tag{F1*curl}
\end{equation}
which lead to reduced formulations analogous to the previous ones with
appropriate modification of the matrix $\sfA$.

\begin{remark}
Sometimes formulation~\eqref{eq:F1curl} is presented in the literature with a
different choice of functional spaces, that is
$\{\bfsigma,\bftau\}\in\mathbf{H}^1(\Omega)$ instead of
$\{\bfsigma,\bftau\}\in\Hdiv\cap\Hcurl$. Although for smooth domains the two
spaces are the same, this is not the case when singular solutions are
presented, that could be in $\Hdiv\cap\Hcurl$ but not in
$\mathbf{H}^1(\Omega)$.
\end{remark}

In a natural way it is possible to consider the \llstar formulation associated
to the formulation enriched with $\curl\bfsigma$: find
$\bfchi\in\Hdiv\cap\Hcurl$, $p\in H^1_0(\Omega)$, and $z\in H^1(\Omega)$, such
that
\[
\left\{
\aligned
&(\bfchi,\bfxi)+(\div\bfchi,\div\bfxi)+(\curl\bfchi,\curl\bfxi)\\
&\hspace{3cm}-(\grad p,\bfxi)+(\curl z,\bfxi)=0&&
\forall\bfxi\in\Hdiv\cap\Hcurl\\
&-(\bfchi,\grad q)+(\grad p,\grad q)-(\curl z,\grad q)=(f,q)&&\forall q\in
H^1_0(\Omega)\\
&(\bfchi,\curl w)-(\grad p,\curl w)+(\curl z,\curl w)=0&&\forall w\in
H^1(\Omega)
\endaligned
\right.
\]

The corresponding eigenvalue problem is then: find $\lambda\in\RE$ and $p\in
H^1_0(\Omega)$, with $p\ne0$, such that for some $\bfchi\in\Hdiv\cap\Hcurl$
and $z\in H^1(\Omega)$ it holds
\begin{equation}
\left\{
\aligned
&(\bfchi,\bfxi)+(\div\bfchi,\div\bfxi)+(\curl\bfchi,\curl\bfxi)\\
&\hspace{3cm}-(\grad p,\bfxi)+(\curl z,\bfxi)=0&&
\forall\bfxi\in\Hdiv\cap\Hcurl\\
&-(\bfxi,\grad q)+(\grad p,\grad q)-(\curl z,\grad q)=\lambda(p,q)&&
\forall q\in H^1_0(\Omega)\\
&(\bfchi,\curl w)-(\grad p,\curl w)+(\curl z,\curl w)=0&&\forall w\in
H^1(\Omega)
\endaligned
\right.
\label{eq:LL*curl}
\tag{LL*curl}
\end{equation}

The operators structure of this problem is now
\[
\left(
\begin{matrix}
\sfA&\sfB^\top&\sfC^\top\\
\sfB&\sfD&\sfE^\top\\
\sfC&\sfE&\sfF
\end{matrix}
\right)
\left(
\begin{matrix}
\sfx\\
\sfy\\
\sfz
\end{matrix}
\right)
=\mu
\left(
\begin{matrix}
\sfzero&\sfzero&\sfzero\\
\sfzero&\sfM&\sfzero\\
\sfzero&\sfzero&\sfzero
\end{matrix}
\right)
\left(
\begin{matrix}
\sfx\\
\sfy\\
\sfz
\end{matrix}
\right)
\]

\section{Galerkin discretizazion}

We now discuss the Galerkin discretization of the problems we have introduced
in the previous section.

\subsection{Approximation of the FOSLS formulations}

Let $\Sigma_h\subset\Hdiv$ and $U_h\subset H^1_0(\Omega)$ be conforming finite
element spaces. The discretization of~\eqref{eq:F1} reads: find
$\lambda_h\in\RE$ and $u_h\in U_h$ with $u_h\ne0$ such that for some
$\bfsigma_h\in\Sigma_h$ it holds
\begin{equation}
\left\{
\aligned
&(\bfsigma_h,\bftau)+(\div\bfsigma_h,\div\bftau)-(\grad
u_h,\bftau)=-\lambda_h(u_h,\div\bftau)&&\forall\bftau\in\Sigma_h\\
&-(\bfsigma_h,\grad v)+(\grad u_h,\grad v)=0&&\forall v\in U_h
\endaligned
\right.
\label{eq:F1h}
\tag{F1h}
\end{equation}

Analogously, the approximation of~\eqref{eq:F1*} has the following form: find
$\lambda_h\in\RE$ and $\bfsigma_h\in\Sigma_h$ with $\bfsigma_h\ne\mathbf{0}$
such that for some $u_h\in U_h$ it holds
\begin{equation}
\left\{
\aligned
&(\bfsigma_h,\bftau)+(\div\bfsigma_h,\div\bftau)-(\grad
u_h,\bftau)=0&&\forall\bftau\in\Sigma_h\\
&-(\bfsigma_h,\grad v)+(\grad u_h,\grad v)=-\lambda_h(\div\bfsigma_h,v)
&&\forall v\in U_h
\endaligned
\right.
\label{eq:F1*h}
\tag{F1*h}
\end{equation}

After introducing basis functions of $\Sigma_h$ and $U_h$, the matrix
structure of Problems~\eqref{eq:F1h} and~\eqref{eq:F1*h} are the ones already
anticipated in~\eqref{eq:F1matrix} and~\eqref{eq:F1*matrix} and that will be
repeated in the next two propositions, where we characterize their
eigensolutions. We will then show that the relevant eigenmodes of the two
formulations are identical.

Before giving a characterization of the eigenvalues of our discrete
formulation, we discuss in the following remark the solution of (possibly
degenerate) generalized eigenvalue problems.

\begin{remark}

In general our discrete problems have the form of a generalized eigenvalue
problem
\begin{equation}
\mathcal{A}x=\lambda\mathcal{B}x
\label{eq:kernel}
\end{equation}
where the matrices $\mathcal{A}$ and/or $\mathcal{B}$ may be singular. The
solution of this problem satisfies the following properties.

\begin{enumerate}

\item If the matrix $\mathcal{B}$ is invertible, then~\eqref{eq:kernel} is
equivalent to the standard eigenvalue problem
$\mathcal{B}^{-1}\mathcal{A}x=\lambda x$.

\item If $\mathcal{K}=\ker\mathcal{A}\cap\ker\mathcal{B}$ is not trivial then
the eigenvalue problem is degenerate and vectors in $\mathcal{K}$ do not
correspond to any eigenvalue of~\eqref{eq:kernel}.

\item If the matrix $\mathcal{B}$ has a non-trivial kernel $\ker(\mathcal{B})$
which does not contain any nonzero vector of $\ker(\mathcal{A})$ then it is
conventionally assumed that~\eqref{eq:kernel} has an eigenvalue
$\lambda=\infty$ with eigenspace equal to $\ker(\mathcal{B})$.

\item If $\mathcal{B}$ is singular and $\mathcal{A}$ is not (which is the most
common situation in our framework) then it may be convenient to switch the
roles of the two matrices and to consider the problem
\[
\mathcal{B}x=\mu\mathcal{A}x
\]
Then $(\mu,x)$ with $\mu=0$ corresponds to the eigenmode $(\infty,x)$
of~\eqref{eq:kernel}; the remaining eigenmodes are $(\lambda,x)$ with
$\lambda=1/\mu$.

\end{enumerate}

\end{remark}

The next proposition is related to the eigensolutions to~\eqref{eq:F1h}.

\begin{proposition}

Let us consider the following matrices associated with Problem~\eqref{eq:F1h}.
\begin{itemize}
\item $\sfA$ is the matrix associated with the bilinear form
$(\bfsigma,\bftau)+(\div\bfsigma,\div\bftau)$, 
\item $\sfB$ is the matrix associated with the bilinear form
$-(\bfsigma,\grad v)$,
\item $\sfC$ is the matrix associated with the bilinear form
$(\grad u,\grad v)$.
\end{itemize}
Then the following generalized problem (see~\eqref{eq:F1matrix})
\[
\left(
\begin{matrix}
\sfA&\sfB^\top\\
\sfB&\sfC
\end{matrix}
\right)
\left(
\begin{matrix}
\sfx\\\sfy
\end{matrix}
\right)=
\lambda_h\left(
\begin{matrix}
\sfzero&-\sfB^\top\\
\sfzero&\sfzero
\end{matrix}
\right)
\left(
\begin{matrix}
\sfx\\\sfy
\end{matrix}
\right)
\]
has three families of eigenvalues. More precisely:

\begin{enumerate}

\item $\lambda_h=\infty$ with multiplicity equal to $\dim\Sigma_h$,

\item $\lambda_h=\infty$ with multiplicity equal to $\dim\ker(B^\top)$,

\item a number of positive eigenvalues $\lambda_h$ (counted with their
multiplicities) equal to $\rank(B^\top)$.

\end{enumerate}
\label{pr:eigF1h}
\end{proposition}

\begin{proof}
The dimension of the eigenproblem is $\dim\Sigma_h+\dim U_h$ which is clearly
equal to the number of eigenvalues in the three families since
$\dim U_h=\dim\ker(B^\top)+\rank(B^\top)$.

The eigenvalues of the first and of the second family are associated to
eigenvectors in the kernel of the matrix on the right hand side. Those are of
the form $(\sfx,\sfy)^\top$ with $\sfx$ corresponding to any element in
$\Sigma_h$ and $\sfy$ corresponding to elements of $U_h$ in $\ker(B^\top)$.

The eigenvalues of the third family are characterized by looking at the Schur
complement
\[
\sfA\sfx=(\lambda_h+1)\sfB^\top\sfC^{-1}\sfB\sfx
\]
\end{proof}

The following proposition is related to the eigensolutions to~\eqref{eq:F1*h}.

\begin{proposition}
Let $\sfA$, $\sfB$, and $\sfC$ be the matrices introduced in
Proposition~\ref{pr:eigF1h}. Then the following generalized eigenvalue
problem associated with Problem~\eqref{eq:F1*h}
\[
\left(
\begin{matrix}
\sfA&\sfB^\top\\
\sfB&\sfC
\end{matrix}
\right)
\left(
\begin{matrix}
\sfx\\\sfy
\end{matrix}
\right)=
\lambda_h\left(
\begin{matrix}
\sfzero&\sfzero\\
-\sfB&\sfzero
\end{matrix}
\right)
\left(
\begin{matrix}
\sfx\\\sfy
\end{matrix}
\right)
\]
has three families of eigenvalues. More precisely:

\begin{enumerate}

\item $\lambda_h=\infty$ with multiplicity $\dim U_h$,

\item $\lambda_h=\infty$ with multiplicity $\dim\ker(B)$,

\item a number of positive eigenvalues $\lambda_h$ (counted with their
multiplicities) equal to $\rank(B)$.

\end{enumerate}
\label{pr:eigF1*h}
\end{proposition}

\begin{proof}
The proof is analogous to the one of Proposition~\ref{pr:eigF1h} by
considering the corresponding Schur complement
\[
\sfC\sfy=(\lambda_h+1)\sfB\sfA^{-1}\sfB^\top\sfy
\]
\end{proof}

Since we started from a self-adjoint problem, it is not surprising that
formulations~\eqref{eq:F1h} and~\eqref{eq:F1*h} are indeed equivalent. This
will be shown in the next proposition.

\begin{proposition}

The eigenmodes of the third families in Propositions~\eqref{pr:eigF1h}
and~\eqref{pr:eigF1*h} are the same.
\label{pr:equiv}
\end{proposition}

\begin{proof}
Solving the matrix formulation of~\eqref{eq:F1h} (see~\eqref{eq:F1matrix} and
Proposition~\ref{pr:eigF1h}) for $\sfx$ gives
$\sfx=-\sfA^{-1}(\lambda_h\sfB^\top\sfy+\sfB^\top\sfy)$, yielding
\[
\sfC\sfy=(\lambda_h+1)\sfB\sfA^{-1}\sfB^\top\sfy,
\]
that is, the Schur complement of the matrix formulation of~\eqref{eq:F1*h}
(see the proof of Proposition~\ref{pr:eigF1*h}).
\end{proof}

We conclude this section with another equivalent matrix formulation
of~\eqref{eq:F1h} and~\eqref{eq:F1*h}.

Starting from
\[
\left(
\begin{matrix}
\sfA&\sfB^\top\\
\sfB&\sfC
\end{matrix}
\right)
\left(
\begin{matrix}
\sfx\\\sfy
\end{matrix}
\right)=
\lambda_h\left(
\begin{matrix}
\sfzero&-\sfB^\top\\
\sfzero&\sfzero
\end{matrix}
\right)
\left(
\begin{matrix}
\sfx\\\sfy
\end{matrix}
\right)
\]
we get
\[
\left(
\begin{matrix}
\sfA&\sfzero\\
\sfB&\sfC
\end{matrix}
\right)
\left(
\begin{matrix}
\sfx\\\sfy
\end{matrix}
\right)=
(\lambda_h+1)\left(
\begin{matrix}
\sfzero&-\sfB^\top\\
\sfzero&\sfzero
\end{matrix}
\right)
\left(
\begin{matrix}
\sfx\\\sfy
\end{matrix}
\right)
\]
and
\[
\left(
\begin{matrix}
\sfA&\sfzero\\
(\lambda_h+1)\sfB&(\lambda_h+1)\sfC
\end{matrix}
\right)
\left(
\begin{matrix}
\sfx\\\sfy
\end{matrix}
\right)=
(\lambda_h+1)\left(
\begin{matrix}
\sfzero&-\sfB^\top\\
\sfzero&\sfzero
\end{matrix}
\right)
\left(
\begin{matrix}
\sfx\\\sfy
\end{matrix}
\right)
\]
leading finally to
\[
\left(
\begin{matrix}
\sfA&\sfzero\\
\sfzero&\sfzero
\end{matrix}
\right)
\left(
\begin{matrix}
\sfx\\\sfy
\end{matrix}
\right)=
(\lambda_h+1)\left(
\begin{matrix}
\sfzero&-\sfB^\top\\
-\sfB&-\sfC
\end{matrix}
\right)
\left(
\begin{matrix}
\sfx\\\sfy
\end{matrix}
\right)
\]

\begin{remark}
The analysis presented in this section applies without modifications to the
formulations enriched with the curl. The only change is the definition of
the matrix $\sfA$ which corresponds to the bilinear form
$(\bfsigma,\bftau)+(\div\bfsigma,\div\bftau)+(\curl\bfsigma,\curl\bftau)$.
\end{remark}

\subsection{Approximation of the \llstar formulation}

The discretization of the \llstar formulation~\eqref{eq:LL*} is obtained after
introducing discrete spaces $\Sigma_h\subset\Hdiv$ and $U_h\subset
H^1_0(\Omega)$. The discrete problem is: find $\mu_h\in\RE$ and $p_h\in U_h$,
with $p_h\ne0$, such that for some $\bfchi_h\in\Sigma_h$ it holds
\begin{equation}
\left\{
\aligned
&(\bfchi_h,\bfxi)+(\div\bfchi_h,\div\bfxi)-(\grad
p_h,\bfxi)=0&&\forall\bfxi\in\Sigma_h\\
&-(\bfchi_h,\grad q)+(\grad p_h,\grad q)=\mu_h(p_h,q)&&\forall q\in U_h
\endaligned
\right.
\label{eq:LL*h}
\tag{LL*h}
\end{equation}
As already observed, this problem is symmetric and it can be written in the
following matrix form
\begin{equation}
\left(
\begin{matrix}
\sfA&\sfB^\top\\
\sfB&\sfC
\end{matrix}
\right)
\left(
\begin{matrix}
\sfx\\
\sfy
\end{matrix}
\right)
=\mu_h
\left(
\begin{matrix}
\sfzero&\sfzero\\
\sfzero&\sfM
\end{matrix}
\right)
\left(
\begin{matrix}
\sfx\\
\sfy
\end{matrix}
\right)
\label{eq:LL*matrix}
\end{equation}
after introducing in a natural way the following matrices:
\begin{itemize}
\item $\sfA$ associated with the bilinear form
$(\bfchi,\bfxi)+(\div\bfchi,\div\bfxi)$,
\item $\sfB$ associated with the bilinear form $-(\bfchi,\grad q)$,
\item $\sfC$ associated with the bilinear form $(\grad p,\grad q)$,
\item $\sfM$ associated with the bilinear form $(p,q)$.
\end{itemize}

The Schur complement associated with the \llstar formulation is easily seen to
be equal to
\[
(-\sfB\sfA^{-1}\sfB^\top+\sfC)\sfy=\mu_h\sfM\sfy
\]

The next proposition, whose proof is immediate, characterizes the eigenvalues
of the \llstar formulation.

\begin{proposition}
The generalized eigenvalue problem~\eqref{eq:LL*matrix} has the following two
families of eigensolutions:
\begin{enumerate}
\item $\mu_h=+\infty$ with multiplicity equal to $\dim\Sigma_h$,
\item a number of positive eigenvalues $\mu_h$ equal to $\dim U_h$.
\end{enumerate}
\end{proposition}

\section{Convergence analysis}

The convergence analysis of the proposed schemes can be performed within the
standard abstract setting presented in~\cite{BaOs} (see also~\cite{acta}).
We first consider the convergence of the eigenmodes (and absence of spurious
modes), then we discuss the rate of convergence.

\subsection{Analysis of the FOSLS formulations}

We start with the analysis of the first formulation that we have considered
in~\eqref{eq:F1}. Thanks to the equivalence shown in
Proposition~\ref{pr:equiv}, the same analysis applies to
formulation~\eqref{eq:F1*} as well.

We introduce a suitable solution operator $\Tf:L^2(\Omega)\to L^2(\Omega)$
associated with the FOSLS formulation presented in~\eqref{eq:F1}. Given $f\in
L^2(\Omega)$ we define $\Tf f\in H^1_0(\Omega)$ as the second component of the
solution of~\eqref{eq:F1source}, so that it solves the following problem for
some $\bfsigma\in\Hdiv$:
\[
\left\{
\aligned
&(\bfsigma,\bftau)+(\div\bfsigma,\div\bftau)-(\grad\Tf
f,\bftau)=-(f,\div\bftau)&&\forall\bftau\in\Hdiv\\
&-(\bfsigma,\grad v)+(\grad\Tf f,\grad v)=0&&\forall v\in H^1_0(\Omega)
\endaligned
\right.
\]

It is easily seen that the operator $\Tf$ is compact (its range is included in
$H^1_0(\Omega)$ which is compact in $L^2(\Omega)$ and self-adjoint (it is the
solution operator associated with the Laplace problem). We enumerate the
reciprocals of its non-vanishing eigenvalues in increasing order so that they
form a sequence tending to $+\infty$
\[
0<\lambda_1\le\lambda_2\le\dots\le\lambda_i\le\cdots
\]
The corresponding eigenfunctions are denoted by $\{u_i\}$,
$i=1,2,\dots,i,\dots$. We consider eigenfunctions normalized in $L^2(\Omega)$
and we repeat the $\lambda_i$'s according to their multiplicities.

Let $\Sigma_h\subset\Hdiv$ and $U_h\subset H^1_0(\Omega)$ be conforming finite
element spaces.
The discrete counterpart of $\Tf$ is the operator $\Tfh:L^2(\Omega)\to
L^2(\Omega)$ defined as follows. Given $f\in L^2(\Omega)$, we define $\Tfh
f\in U_h$ as the second component of the solution of the Galerkin
approximation of~\eqref{eq:F1source}, so that is solves the following problem
for some $\bfsigma_h\in\Sigma_h$:
\[
\left\{
\aligned
&(\bfsigma_h,\bftau)+(\div\bfsigma_h,\div\bftau)-(\grad\Tfh
f,\bftau)=-(f,\div\bftau)&&\forall\bftau\in\Sigma_h\\
&-(\bfsigma_h,\grad v)+(\grad\Tfh f,\grad v)=0&&\forall v\in U_h
\endaligned
\right.
\]
Since $U_h$ is finite dimensional, the operator $\Tfh$ is compact; moreover,
it is self-adjoint (see, for instance, all equivalent matrix characterizations
presented in the previous section).
We denote the reciprocals of its non-vaninshing eigenvalues in analogy to what
we have done for the continuous operator $\Tf$:
\[
0<\lambda_{1,h}\le\lambda_{2,h}\le\dots\le\lambda_{i,h}\le\cdots\le\lambda_{N(h),h},
\]
where $N(h)\le\dim(U_h)$ is the rank of the matrix $\sfB^\top$ in
Proposition~\ref{pr:eigF1h}.
The corresponding eigenfunctions are denoted by $\{u_{i,h}\}$,
$i=1,2,\dots,N(h)$, with the same convention for normalization and multiple
eigenvalues.

We summarize in the following proposition what is needed in order to show the
convergence of the discrete eigenmodes to the continuous ones (see~\cite{BaOs}
and~\cite{acta}).

\begin{proposition}
Let us assume that the operator sequence $\Tfh$ converges in norm to $\Tf$ as
$h$ goes to zero, that is,
\begin{equation}
\|\Tf f-\Tfh f\|_0\le\rho(h)\|f\|_0\quad\forall f\in L^2(\Omega)
\label{eq:cunif}
\end{equation}
with $\rho(h)$ tending to zero as $h$ goes to zero. Let
$\lambda_i=\lambda_{i+1}=\dots=\lambda_{i+m-1}$ be an eigenvalue of multiplicity
$m$ associated with the operator $\Tf$. Then, for $h$ small enough, so that
$N(h)\ge i+m-1$, the $m$ discrete eigenvalues $\lambda_{j,h}$
($j=i,\dots,i+m-1$) associated with the operator $\Tfh$ converge to
$\lambda_i$.
Moreover, the corresponding eigenfunctions converge, that is
\[
\delta(E,E_h)\to 0\quad\text{as $h$ goes to zero},
\]
where $\delta$ denote as usual the gap between Hilbert subspaces, $E$ is the
continuous eigenspace spanned by $\{u_i,\dots,u_{i+m-1}\}$, and $E_h$ is its
discrete counterpart spanned by $\{u_{i,h},\dots,u_{i+m-1,h}\}$.
\label{pr:cunif}
\end{proposition}

We recall the standard a priori error estimate for the solution of the source
problem~\eqref{eq:F1}. It follows with standard arguments since the
formulation is coercive that we have
\begin{equation}
\|\bfsigma-\bfsigma_h\|_{\Hdiv}+\|u-u_h\|_{H^1}\le C
\inf_{\substack{\bftau_h\in\Sigma_h\\ v_h\in U_h}}
(\|\bfsigma-\bftau_h\|_{\Hdiv}+\|u-v_h\|_{H^1})
\label{eq:apriori}
\end{equation}

Let us assume that the domain is a Lipschitz polygon/polyhedron, then we know
that if $f$ is in $L^2(\Omega)$ then the solution $u$ belongs to
$H^{1+s}(\Omega)$ for some $s\in(1/2,1]$.

Unfortunately, estimate~\eqref{eq:apriori} is not enough to obtain the uniform
convergence~\eqref{eq:cunif} of $\Tfh$ to $\Tf$.
Take, for instance, standard finite element spaces, so that the best
approximation properties on the right hand side of~\eqref{eq:apriori} read as
follows
\[
\aligned
&\inf_{\bftau_h\in\Sigma_h}\|\bfsigma-\bftau_h\|_{\Hdiv}\le
Ch^s\|\bfsigma\|_{\bfH^{1+s}}\\
&\inf_{v_h\in U_h}\|u-v_h\|_{H^1}\le Ch^s\|u\|_{H^{1+s}}
\endaligned
\]
Clearly, the regularity of $\bfsigma$ is not enough to guarantee a rate of
convergence, since $\div\bfsigma=-f$ cannot be assumed more regular than
$L^2(\Omega)$, whence $\bfsigma$ in general is not in $\bfH^{1+s}$.

The approximation of $\bfsigma$ could be improved when using more natural
discretization of $\Hdiv$, such as the Raviart--Thomas spaces, as
follows:
\[
\inf_{\bftau_h\in\Sigma_h}\|\bfsigma-\bftau_h\|_{\Hdiv}\le
Ch^s(\|\bfsigma\|_{\bfH^s}+\|\div\bfsigma\|_{H^s})
\]
However, also in this case we see that we cannot get a rate of convergence out
of this estimate for the same reason as before.

What we have observed is a well known fact due to the lack of compactness of
the problem we are studying, when considered in terms of both component of the
solution.

On the other hand, the a priori estimate~\eqref{eq:apriori} is a very strong
result, since it involves the error in the $\Hdiv$ norm of $\bfsigma$ and the
error in the $H^1(\Omega)$ norm of $u$ combined together. For the uniform
convergence it is enough to estimate the error in the $L^2(\Omega)$ of the
only component $u$. This can be done by using a standard duality argument and
the corresponding result is stated in the next lemma.

\begin{lemma}
Let $u\in H^{1+s}(\Omega)$ ($s>1/2$) be the second component of the solution
to~\eqref{eq:F1source} and $u_h\in U_h$ the corresponding numerical solution.
Assume that the finite element spaces $\Sigma_h$ and $U_h$ satisfy the
following approximation properties
\[
\aligned
&\inf_{\bftau\in\Sigma_h}\|\bfchi-\bftau\|_{\Hdiv}\le
Ch^s\|\bfchi\|_{\bfH^s(\Omega)}+\|\div\bfchi\|_{H^{1+s}(\Omega)}\\
&\inf_{v\in U_h}\|p-v\|_{H^1(\Omega)}\le Ch^s\|p\|_{H^{1+s}(\Omega)}
\endaligned
\]
Then the following estimate holds true
\[
\|u-u_h\|_{L^2(\Omega)}\le
Ch^s(\|\bfsigma-\bfsigma_h\|_{\Hdiv}+\|u-u_h\|_{H^1})
\]
\label{le:precunif}
\end{lemma}

\begin{proof}
This proof has been essentially already presented in~\cite[Sec.~7]{vecquad} in
a different context for convex domains (see also~\cite{CaiKu}).

We aim at providing a refined $L^2$ estimate of the error $\|u-u_h\|$ of the
formulation~\eqref{eq:F1source} and of its corresponding
discretization (with appropriate choice of the finite element spaces).
The error will be estimated in terms of the natural error
$\|\bfsigma-\bfsigma_h\|_{\Hdiv}+\|u-u_h\|_{H^1}$.

We consider the following dual problem (which is pretty much related to the
formulation~\eqref{eq:LL*source}): find $\bfchi\in\Hdiv$ and
$p\in H^1_0(\Omega)$ such that
\begin{equation}
\left\{
\aligned
&(\bfchi,\bfxi)+(\div\bfchi,\div\bfxi)-(\grad
p,\bfxi)=0&&\forall\bfxi\in\Hdiv\\
&-(\bfchi,\grad q)+(\grad p,\grad q)=(u-u_h,q)&&\forall q\in H^1_0(\Omega)
\endaligned
\right.
\label{eq:dualABF}
\end{equation}

If the domain is convex (or in general if the domain is smooth enough so that
the Poisson problem has $H^2$ regularity), the solution of the above problem
satisfies
\begin{equation}
\aligned
&\bfchi=\grad(p+g)&&\text{with $g\in H^2(\Omega)\cap H^1_0(\Omega)$ s.t.}\\
&\Delta g=u-u_h\\
&\Delta p=g-u+u_h
\endaligned
\label{eq:miracle}
\end{equation}
so that, in particular, $\div\bfchi=g$; moreover, the following stability
bound is valid:
\[
\|p\|_{H^2}+\|\bfchi\|_{H^1}+\|\div\bfchi\|_{H^2}\le C\|u-u_h\|_{L^2}
\]

In the case of the regularity assumed in our case ($s>1/2$) we have
that~\eqref{eq:miracle} is valid in variational form with $g\in
H^{1+s}(\Omega)\cap H^1_0(\Omega)$ and we obtain the following bound:
\begin{equation}
\|p\|_{H^{1+s}}+\|\bfchi\|_{H^s}+\|\div\bfchi\|_{H^{1+s}}\le C\|u-u_h\|_{L^2}
\label{eq:bound_opt}
\end{equation}

Taking as test functions in~\eqref{eq:dualABF} $\bfxi=\bfsigma-\bfsigma_h$ and
$q=u-u_h$ in~\eqref{eq:dualABF}, summing the two equations, and using the
error equations related to~\eqref{eq:F1source} and its discretization, we
obtain
\[
\aligned
\|u-u_h\|^2_{L^2}={}&(\bfchi,\bfsigma-\bfsigma_h)
+(\div\bfchi,\div(\bfsigma-\bfsigma_h))
-(\grad p,\bfsigma-\bfsigma_h)\\
&-(\bfchi,\grad(u-u_h))+(\grad p,\grad(u-u_h))\\
={}&(\bfchi-\bftau_h,\bfsigma-\bfsigma_h)
+(\div(\bfchi-\bftau_h),\div(\bfsigma-\bfsigma_h))\\
&-(\grad(p-v_h),\bfsigma-\bfsigma_h)\\
&-(\bfchi-\bftau_h,\grad(u-u_h))+(\grad(p-v_h),\grad(u-u_h))
\endaligned
\]
for all $\bftau_h\in\Sigma_h$ and $v_h\in U_h$.

It follows
\[
\|u-u_h\|^2_{L^2}\le C(\|\bfchi-\bftau_h\|_{\Hdiv}+\|p-v_h\|_{H^1})
(\|\bfsigma-\bfsigma_h\|_{\Hdiv}+\|u-u_h\|_{H^1})
\]

Using the approximation estimates assumed for $\Sigma_h$ and $U_h$ and the
bound in~\eqref{eq:bound_opt} we finally get
\[
\|u-u_h\|_{L^2}\le Ch^s(\|\bfsigma-\bfsigma_h\|_{\Hdiv}+\|u-u_h\|_{H^1})
\]
\end{proof}

The results of the previous lemma gives directly the uniform convergence that
implies the convergence of the eigenvalues according to
Proposition~\ref{pr:cunif}.

\begin{theorem}
Under the same hypotheses as in Lemma~\ref{le:precunif} the uniform
convergence~\eqref{eq:cunif} holds true.
\end{theorem}

\begin{proof}
We have
\[
\aligned
\|\Tf f-\Tfh f\|_{L^2}={}&\|u-u_h\|_{L^2}
\le Ch^s (\|\bfsigma-\bfsigma_h\|_{\Hdiv}+\|u-u_h\|_{H^1})\\
&\le Ch^s \|f\|_{L^2}
\endaligned
\]
\end{proof}

Let us now move to the analysis of the rate of convergence.

We start with the estimate of the eigenfunctions. Standard Babu\v ska--Osborn
theory (see~\cite{BaOs} or~\cite[Th.~9.10]{acta}) implies the following
result.
\begin{proposition}
Let $\lambda_i=\lambda_{i+1}=\dots=\lambda_{i+m-1}$ be an eigenvalue of
multiplicity $m$ and denote by $E=\Span\{u_i,\dots,u_{i+m-1}\}$ the
corresponding eigenspace. Then
\begin{equation}
\delta(E,E_h)\le C\|(\Tf-\Tfh)|_E\|_{\mathcal{L}(H^1)},
\label{eq:efunest}
\end{equation}
where $E_h=\Span\{u_{i,h},\dots,u_{i+m-1,h}\}$ is the space generated by the
corresponding discrete eigenfunctions.
\end{proposition}

In order to bound the right hand side in~\eqref{eq:efunest} we can use the
standard energy norm estimate for~\eqref{eq:F1source} which reads
\[
\|\bfsigma-\bfsigma_h\|_{\Hdiv}+\|u-u_h\|_{H^1}\le
\inf_{\substack{\bftau\in\Sigma_h\\ v\in U_h}}
(\|\bfsigma-\bftau\|_{\Hdiv}+\|u-v\|_{H^1})
\]

The final estimate is summarized in the following theorem.

\begin{theorem}
Let $\lambda_i=\lambda_{i+1}=\dots=\lambda_{i+m-1}$ be an eigenvalue of
multiplicity $m$; denote by $E=\Span\{u_{i},\dots,u_{i+m-1}\}$ its eigenspace
and by $E_h=\Span\{u_{i,h},\dots,u_{i+m-1,h}\}$ the space generated by the
corresponding discrete eigenfunctions. Then for all $j=i,\dots,i+m-1$ there
exists $u_h\in E_h$ such that
\begin{equation}
\|u_j-u_h\|_{H^1}\le
C\sup_{\substack{u\in E\\\|u\|=1}}\inf_{\substack{\bftau\in\Sigma_h\\ v\in U_h}}
(\|\grad u-\bftau\|_{\Hdiv}+\|u-v\|_{H^1})
\label{eq:efunest2}
\end{equation}
\label{th:f1}
\end{theorem}

Once we have the optimal estimate for the eigenfunctions, it is
straightforward to obtain the analogous optimal estimate for the eigenvalues.
In this case, since we have seen that our formulation is symmetric (see for
instance the Schur complement formulation~\eqref{eq:matr1}), we obtain as
usual double order of convergence.

\begin{theorem}
Let $\lambda_i=\lambda_{i+1}=\dots=\lambda_{i+m-1}$ be an eigenvalue of
multiplicity $m$ and denote by $\epsilon_{\lambda}(h)$ the quantity appearing
on the right hand side of estimate~\eqref{eq:efunest2}. Then
\[
|\lambda-\lambda_j|\le C\epsilon_{\lambda}(h)^2\qquad\forall j=i,\dots,i+m-1
\]
\label{th:f2}
\end{theorem}

\begin{remark}
One of the most commonly used scheme used for the approximation
of~\eqref{eq:F1source}, based on Ravart--Thomas spaces, is $RT_{k-1}-P_k$
($k\ge1$). In this case the rate of convergence predicted
by~\eqref{eq:efunest2} is $O(h^k)$ provided $u$ belongs to $H^{k+1}(\Omega)$.
In particular, for the lowest order choice, $u\in H^2(\Omega)$ implies first
order convergence $O(h)$ for the eigenfunctions and second order convergence
$O(h^2)$ for the eigenvalues.
\end{remark}

\begin{remark}
If standard (nodal) finite elements are used for the definition of $\Sigma_h$,
then the approximation properties assumed in Lemma~\ref{le:precunif} are not
valid anymore. It is not clear in this case if the uniform
convergence~\eqref{eq:cunif} is satisfied and if the eigenmodes are well
approximated. We are going to present some numerical experiments in
Section~\ref{se:num} where it is shown that the method seems to work in simple
cases.
\end{remark}

\subsection{Analysis of the \llstar formulation}

The analysis of the convergence for the \llstar formulation can be performed
in a similar way as for the FOSLS formulation. We consider the solution
operator $\Tl$ associated with the \llstar formulation: $\Tl f\in
H^1_0(\Omega)$ solves the following problem for some $\bfchi\in\Hdiv$
\[
\left\{
\aligned
&(\bfchi,\bfxi)+(\div\bfchi,\div\bfxi)-(\grad
\Tl f,\bfxi)=0&&\forall\bfxi\in\Hdiv\\
&-(\bfchi,\grad q)+(\grad\Tl f,\grad q)=(f,q)&&\forall q\in H^1_0(\Omega)
\endaligned
\right.
\]

The corresponding discrete operator $\Tlh$ is defined by $\Tlh f\in U_h$ that
solves the following problem for some $\bfchi_h\in\Sigma_h$
\[
\left\{
\aligned
&(\bfchi_h,\bfxi)+(\div\bfchi_h,\div\bfxi)-(\grad
\Tlh f,\bfxi)=0&&\forall\bfxi\in\Sigma_h\\
&-(\bfchi_h,\grad q)+(\grad\Tlh f,\grad q)=(f,q)&&\forall q\in U_h
\endaligned
\right.
\]

As for the FOSLS formulation the uniform convergence of $\Tlh$ to $\Tl$ is
related to an $L^2(\Omega)$ estimate for the \llstar formulation that can be
derived by using a duality argument which makes use of the following auxiliary
problem: find $\tilde\bfchi\in\Hdiv$ and $\tilde p\in H^1_0(\Omega)$ such that
\[
\left\{
\aligned
&(\tilde\bfchi,\bfxi)+(\div\tilde\bfchi,\div\bfxi)-(\grad
\tilde p,\bfxi)=0&&\forall\bfxi\in\Hdiv\\
&-(\tilde\bfchi,\grad q)+(\grad\tilde p,\grad q)=(\Tl f-\Tlh f,q)&&
\forall q\in H^1_0(\Omega)
\endaligned
\right.
\]
Then the following theorem can be proved as in Lemma~\ref{le:precunif}.
\begin{theorem}
Let us assume the same regularity for the solution of our problem as in
Lemma~\ref{le:precunif}. Then the following uniform convergence holds true
\[
\|\Tl f-\Tlh f\|_{L^2(\Omega)}\le\rho(h)\|f\|_{L^(\Omega)}
\]
where $\rho(h)$ tends to zero as $h$ goes to zero.
\end{theorem}

\begin{remark}
Using the previous theorem and the abstract results about the approximation of
eigenvalue problems (see Proposition~\ref{pr:cunif}, and~\cite{BaOs,acta}),
together with the equivalence stated in Proposition~\ref{pr:ll*}, Theorems
analogous to~\ref{th:f1} and to~\ref{th:f2} can be obtained.
\end{remark}

\subsection{Remarks on the formulation enriched with $\curl\bfsigma$}
\label{ss:dauge}

In this section we recall some issues related to the formulations presented in
Subsection~\ref{ss:curl}.

First of all we observe that in this case it is not possible to use
Raviart--Thomas elements for the definition of $\Sigma_h$. Indeed, the
conformity in $\Hdiv$ implies the continuity of the normal trace across
elements (which is compatible with Raviart--Thomas elements), while the
conformity in $\Hcurl$ requires the continuity of the tangential trace. In
practice, if $\Sigma_h$ contains piecewise polynomials, if must be made of
\emph{continuous} elements, so that we have $\Sigma_h\subset\bfH^1(\Omega)$.

A duality argument leading to a refined $\L^2(\Omega)$ estimate for the
div-curl source problem associated with formulation~\eqref{eq:F1curl} was
presented in~\cite{pflaum}. Under certain hypothesis on the domain the
following estimate was shown: there exists $t>1$ such that
\[
\|\bfsigma-\bfsigma_h\|_{\mathbf{L}^2(\Omega)}+\|u-u_h\|_{L^2(\Omega)}\le C
h^{t-1} (\|\bfsigma-\bfsigma_h\|_{\bfH^1(\Omega)}+\|u-u_h\|_{H^1(\Omega)})
\]

On the other hand, in~\cite{costabel} is was shown that the space
$\bfH^1(\Omega)\cap\Hocurl$ is closed in $\Hdiv\cap\Hocurl$. This fact has
negative consequences for the finite element approximation of the solution
of~\eqref{eq:F1curl} and of~\eqref{eq:LL*curl} when $\bfsigma$ does not belong
to $\bfH^1(\Omega)$.
This fact has been observed, in the case of least-squares finite element
methods, in~\cite{Fix-Stephan, Cox-Fix} and later in the case of finite
element approximation of Maxwell's eigenvalues in~\cite{Costabel-Dauge}.

In Section~\ref{se:num} we show an example of bad behavior of the discrete
solution in presence of singularity. We believe that a modification of the
scheme in the spirit of what has been proposed in~\cite{Fix-Stephan, Cox-Fix}
and~\cite{Costabel-Dauge} could lead to good results.

\section{A posteriori analysis}
\label{se:post}

In this section we show how it is possible to define a residual based a
posteriori error estimator and to show its equivalence to the actual error.
For simplicity, we will only discuss the case of the FOSLS
formulation~\eqref{eq:F1} even if analogous constructions can be performed by
the other formulations.

Usually, least-squares finite element formulations come with an intrinsic a
posteriori estimator which is based on the functional used for the definition
of the method. However, in the case of the eigenvalue problem that we
presented, we are computing eigensolutions of the operator associated with
the least-square formulations of the source problem. It follows that the
construction and the analysis of our posteriori error estimator will be
performed in a more conventional way like for standard variational
formuations.

The analysis we are presenting is using arguments that have been already
adopted in the literature for analogous problems. We refer, in particular,
to~\cite{DuranPadraRodriguez} for the approximation of standard Laplace
eigenproblem and to~\cite{Alonso1996,Carstensen1997} for the source Laplace
problem in mixed form. The interested reader is also referred
to~\cite{dietmar} for the Laplace eigenproblem in mixed form.

We consider the following estimator on a single element $T$
\[
\aligned
\eta^2_T&{}=h^2_T\|\div\bfsigma_h-\Delta u_h\|_{L^2(T)}^2+
h^2_T\|\curl\bfsigma\|_{L^2(T)}^2\\
&\quad+\sum_{e\in\partial T}h_e\left(
\|\llbracket\bfsigma\cdot\bft\rrbracket\|_{L^2(e)}^2+
\|\llbracket\grad u_h\cdot\bfn\rrbracket\|_{L^2(e)}^2\right)
\endaligned
\]
which gives as usual the global estimator
\[
\eta_h^2=\sum_T\eta_T^2
\]

The next theorem shows the reliability of the proposed error indicator. For
the sake of readability we state the result in the case of a simple
eigenvalue. More general situations can be handled with standard arguments. We
consider the approximation of~\eqref{eq:F1} where the spaces $\Sigma_h$ and
$U_h$ are one of the standard mixed families (Raviart--Thomas,
Brezzi--Douglas--Marini, etc.) and a standard finite element space of
continuous piecewise polynomials in $H^1_0(\Omega)$, respectively. We do not
impose any condition on the polynomial order of $\Sigma_h$ and $U_h$.
\begin{theorem}[Reliability]
Let $\lambda\in\RE$ be a simple eigenvalue of~\eqref{eq:F1} with eigenfunction
$u\in H^1_0(\Omega)$ and let $\bfsigma\in\Hdiv$ be the other component of the
solution. Consider the approximation $\lambda_h$ of $\lambda$ with
eigenfunction $u_h\in U_h$ converging to $u$ (this can be obtained by
appropriate normalization and choice of the sign) and let
$\bfsigma_h\in\Sigma_h$ be converging analogously to $\bfsigma$.
Then there exists a constant $C$, depending only on the choice of the spaces
$\Sigma_h$ and $U_h$, and on the shape of the elements, such that
\[
\|\bfsigma-\bfsigma_h\|_{\mathbf{L}^2(\Omega)}+\|u-u_h\|_{H^1(\Omega)}\le C
(\eta_h+h\|\div(\bfsigma-\bfsigma_h)\|_{L^2(\Omega)})
\]
\label{le:rel}
\end{theorem}

\begin{proof}
Let us start with the estimate of $\|\bfsigma-\bfsigma_h\|_{L^2(\Omega)}$. We
consider the Helmholtz decomposition of $\bfsigma_h$
\[
\bfsigma_h=\grad\alpha+\curl\beta
\]
with $\alpha\in H^1_0(\Omega)$.
Then we have $\bfsigma-\bfsigma_h=\grad z-\curl\beta$ with $z=u-\alpha$ and
\[
\|\bfsigma-\bfsigma_h\|_{L^2(\Omega)}^2=\|\grad z\|_{L^2(\Omega)}^2+
\|\curl\beta\|_{L^2(\Omega)}^2
\]
It is then standard to estimate $\grad z$ as follows
\[
\aligned
\|\grad z\|_{L^2(\Omega)}^2&{}=(\grad z,\bfsigma-\bfsigma_h)
=-(z,\div(\bfsigma-\bfsigma_h)\\
&=-=(z-z^I,\div(\bfsigma-\bfsigma_h)-(\grad(u-u_h),\grad z^I)\\
&\le Ch\|\grad z\|_{L^2(\Omega)}\|\div(\bfsigma-\bfsigma_h)\|_{\Hdiv}+
\|\grad (u-u_h)\|_{L^2(\Omega)}\|\grad z^I\|_{L^2(\Omega)}
\endaligned
\]
where $z^I$ is an approximation of $z$ in $U_h$ and where we used the error
equation associated with our formulation.

The estimate of $\curl\beta$ is performed as usual by considering the
Scott--Zhang interpolant $\beta^I$ of $\beta$; we observe that we have
\[
(\curl\beta,\curl\beta^I)=-(\bfsigma-\bfsigma_h,\curl\beta^I)=0
\]
Indeed choosing $\bftau=\curl\beta^I$ in the following error equation
\[
(\bfsigma-\bfsigma_h,\bftau)+(\div(\bfsigma-\bfsigma_h),\div\bftau)
-(\grad(u-u_h),\bftau)=(\lambda u-\lambda_hu_h,\div\bftau)
\]
gives
\[
(\bfsigma-\bfsigma_h,\curl\beta^I)-(\grad(u-u_h),\curl\beta^I)=
(\bfsigma-\bfsigma_h,\curl\beta^I)=0
\]
Hence we have
\[
\aligned
\|\curl\beta\|_{L^2(\Omega)}^2&{}=(\curl\beta,\curl(\beta-\beta^I))=
-(\bfsigma-\bfsigma_h,\curl\beta)\\
&=\sum_T\left(\int_T\curl(\bfsigma-\bfsigma_h)(\beta-\beta^I)-
\int_{\partial T}(\bfsigma-\bfsigma_h)\cdot\bft(\beta-\beta^I)\right)\\
&\le C\left(\sum_Th_T\|\curl\bfsigma_h\|_{L^2(T)}+
\sum_eh_e^{1/2}\|\llbracket\bfsigma_h\cdot\bft_e\rrbracket\|_{L^2(e)}\right)
\|\curl\beta\|_{L^2(\Omega)}
\endaligned
\]

Let us now move to the estimate of $\grad(u-u_h)$. We observe that from our
error equation we have
\[
(\grad(u-u_h),v_h)=(\bfsigma-\bfsigma_h,\grad v_h)
\]
for all $v\in U_h$. It follows
\[
\|\grad(u-u_h)\|_{L^2(\Omega)}^2=(\grad(u-u_h),\grad((u-u_h)-w))+
(\bfsigma-\bfsigma_h,\grad w)
\]
where $w\in U_h$ is the Scott--Zhang interpolant of $u-u_h$. The second term
in the last expression can be easily be bounded by
$\|\bfsigma-\bfsigma_h\|_{L^2(\Omega)}\|\grad(u-u_h)\|_{L^2(\Omega)}$, so that
we have to estimate the first one. By standard arguments we have
\[
\aligned
(\grad(u-u_h),\grad((u-u_h)-w))&{}=
\sum_T\Big(-\int_T(\div\bfsigma-\div\grad u_h)((u-u_h)-w)\\
&\quad+\int_{\partial T}\grad u_h\cdot\bfn((u-u_h)-w)\Big)
\endaligned
\]
Hence we have
\[
\aligned
\|\grad(u-u_h)\|_{L^2(\Omega)}^2&{}\le
C\Big(\|\bfsigma-\bfsigma_h\|_{L^2(\Omega)}+
\sum_Th_T\|\div\bfsigma_h-\Delta u_h\|_{L^2(\Omega)}\\
&\quad+\sum_eh_e^{1/2}\|\llbracket\grad u_h\cdot\bfn\rrbracket\|_{L^2(e)}+
h\|\div(\bfsigma-\bfsigma_h)\|_{L^2(\Omega)}\Big)\\
&\quad\|\grad(u-u_h)\|_{L^2(\Omega)}
\endaligned
\]
which together with the obtained estimate for
$\|\bfsigma-\bfsigma_h\|_{L^2(\Omega)}$ implies the result.
\end{proof}

The efficiency of the proposed estimator can be shown as it is standard by
local inverse inequalities and the use of suitable bubble functions.
Without giving any additional detail we state the final result.
\begin{theorem}[Efficiency]
We the same hypotheses as for the reliability result, we have that the error
is an upper bound for our estimator, that is
\[
\eta_h\le C\left(
\|\bfsigma-\bfsigma_h\|_{\Hdiv}+\|u-u_h\|_{H^1(\Omega)}\right)
\]
\end{theorem}

\section{Numerical examples}
\label{se:num}

In this section we report some numerical examples that confirm the theoretical
results of this paper. Moreover, we shall show how the a posteriori analysis
developed in Section~\ref{se:post} can be used in the framework of an adaptive
scheme.

\subsection{A priori convergence: FOSLS formulation}

In order to confirm the convergence rates stated in Theorems~\ref{th:f1}
and~\ref{th:f2} we first consider a square domain $\Omega=]0,1[^2$ where the
solution of the Laplace eigenvalue problems is well known. We compare the
solutions computed with a standard finite element formulation (continuous
Lagrangian elements of order one), a standard mixed finite element formulation
(based on lowest order Raviart--Thomas elements) and the FOSLS
formulation~\eqref{eq:F1h}, where we have made three choices for the space
$\Sigma_h$: Raviart--Thomas element, Brezzi--Douglas--Marini element, and
standard Lagrangian element of lowest order; in all cases we use
continuous piecewise linear polynomials for the space $U_h$ in the FOSLS
formulation. It turns out that the results are pretty much comparable and that
also in the case of the FOSLS formulation with Lagrangian elements, which is
not covered by our theory, we obtain reasonable results.

Figure~\ref{fg:apriori} shows various error quantities related to the
approximation of the smallest eigenvalue with the considered numerical
schemes.

\begin{figure}
\includegraphics[height=6cm]{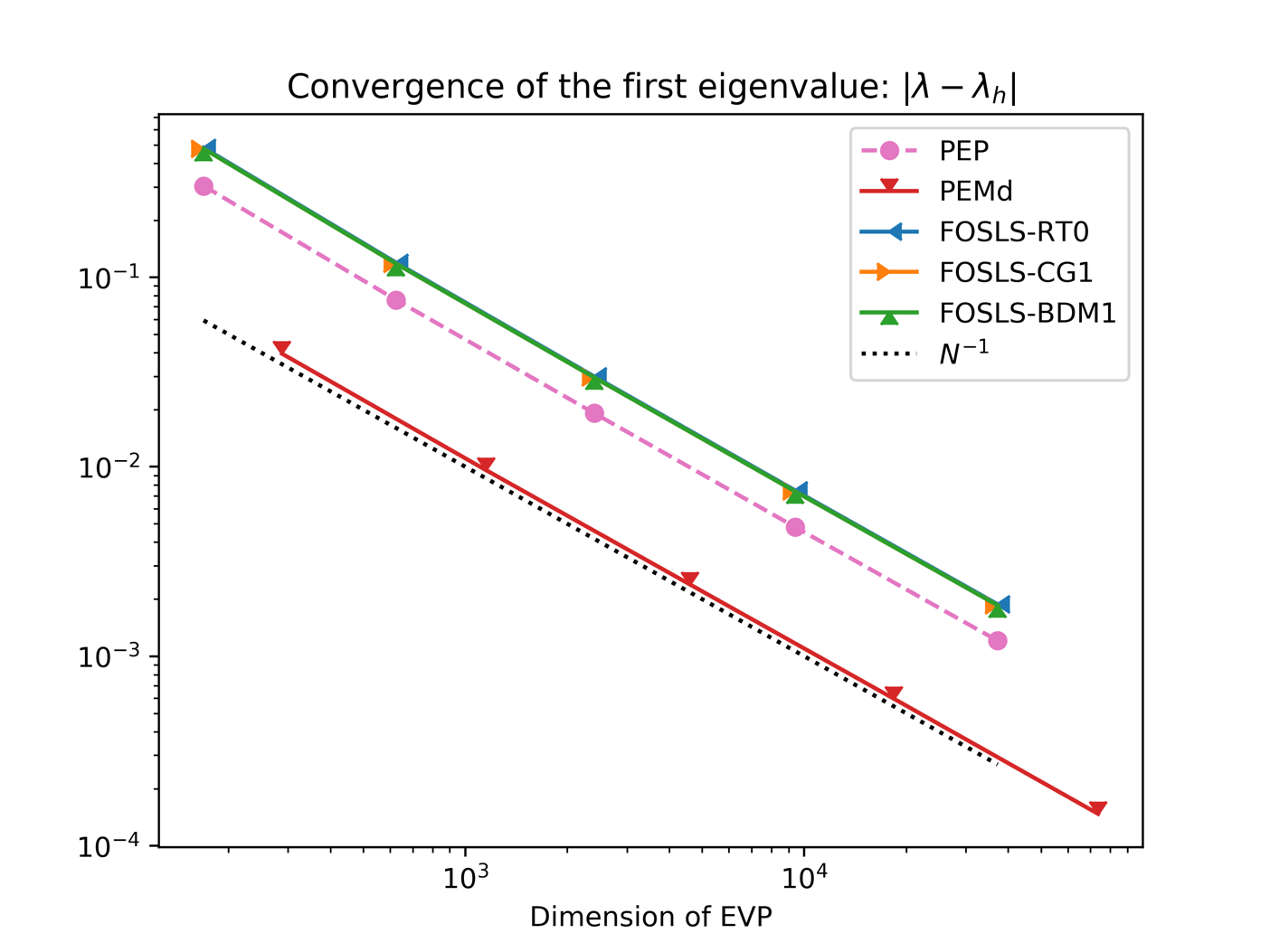}\\
\includegraphics[height=4.5cm]{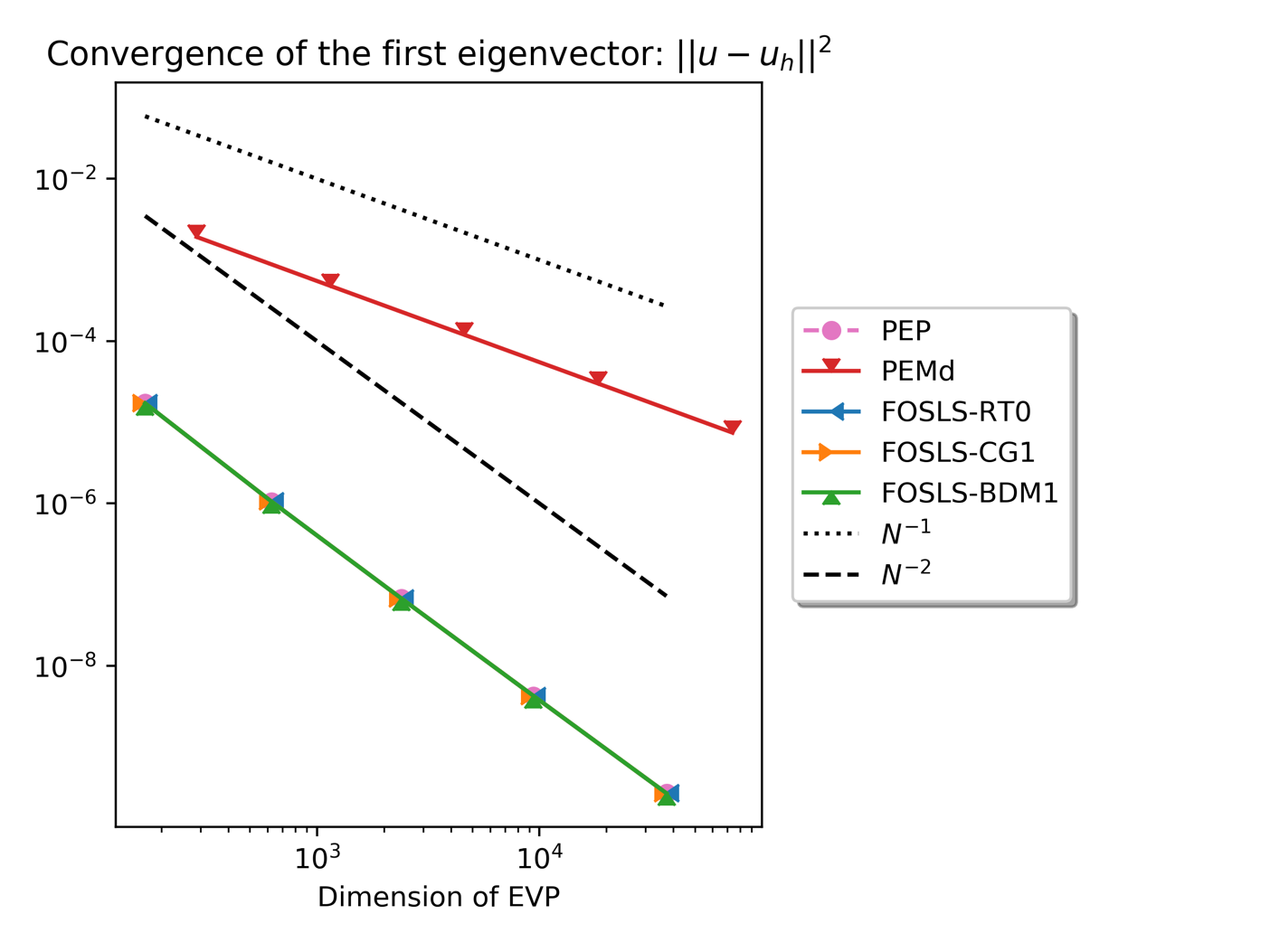}
\includegraphics[height=4.5cm]{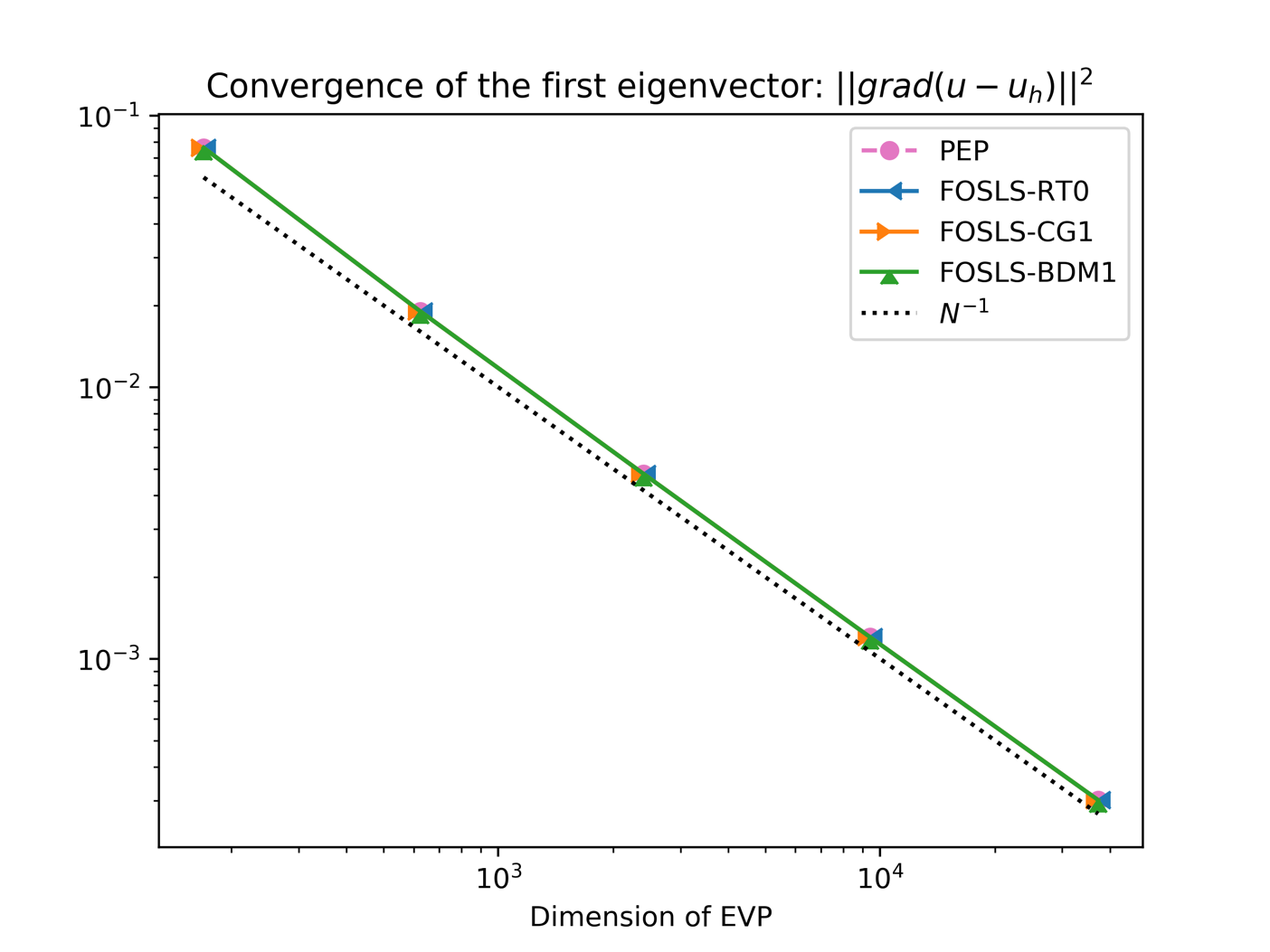}\\
\includegraphics[height=4.5cm]{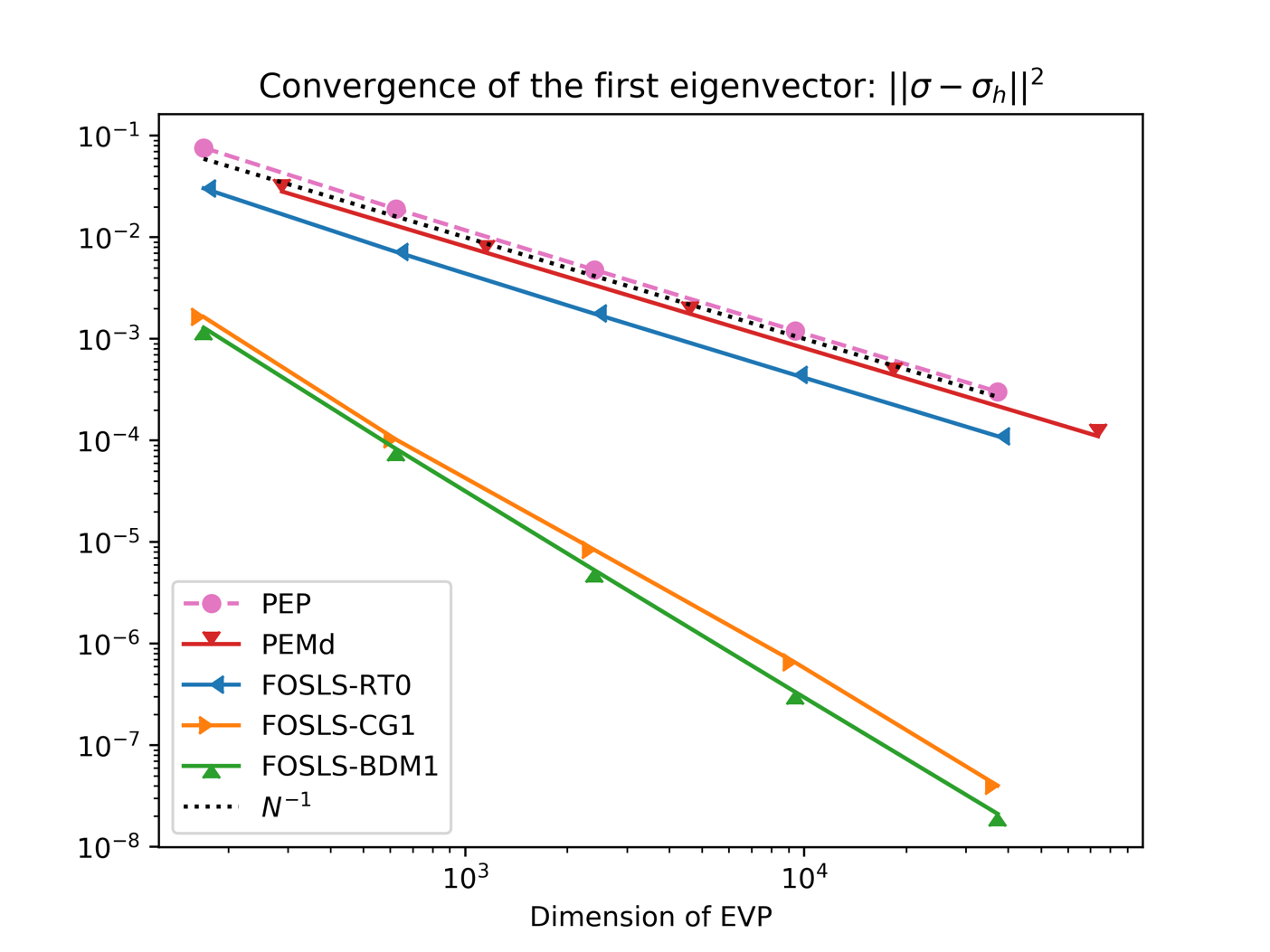}
\includegraphics[height=4.5cm]{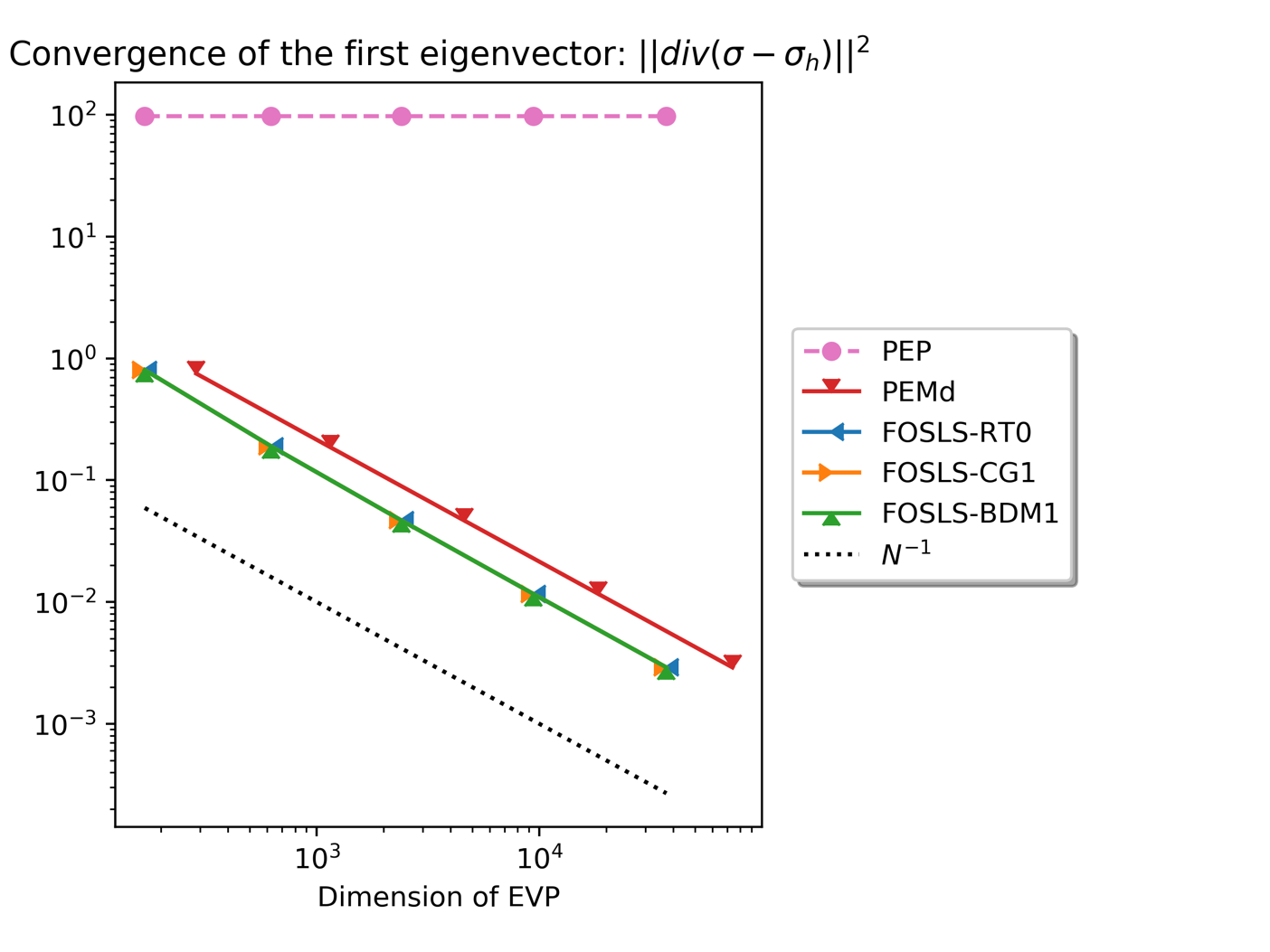}
\caption{Error of the eigenvalue $\lambda$ and error in the $L^2$-norm of $u$,
$\grad u$, $\bfsigma$, and $\div\bfsigma$. The used methods are the standard
Galerkin formulation (PEP), the mixed formulation (PEMd), the FOSLS
formulation with lowest-order Raviart--Thomas (FOSLS-RT0), continuous
Lagrangian (FOSLS-CG1), and Brezzi--Douglas--Marini (FOSLS-BDM1) elements}
\label{fg:apriori}
\end{figure}

\subsection{Formulation enriched with $\curl\bfsigma$}

In Subsection~\ref{ss:curl} we discussed how to enrich the FOSLS formulation
by explicitly imposing that $\curl\bfsigma$ is zero. We observed in
Subsection~\ref{ss:dauge} that the resulting formulation is not expected to
provide good results in presence of solutions where the variable $\bfsigma$
is not sufficiently regular. We computed the eigenvalues of our problem on an
L-shaped domain with continuous piecewise polynomials for both variables. From
the convergence plots, shown in Figure~\ref{fg:Lshaped}, it is clear the first
five eigenvalues have different convergence properties.
In particular, the first (singular) eigenvalue is not converging; it could be
actually shown that it converges optimally towards a wrong value. This is a
similar behavior as what as been previously observed for other formulations
involving $\div$ and $\curl$ of $\bfsigma$ (see, for
instance,~\cite{Fix-Stephan, Cox-Fix, Costabel-Dauge}.
\begin{figure}
\includegraphics[width=8cm]{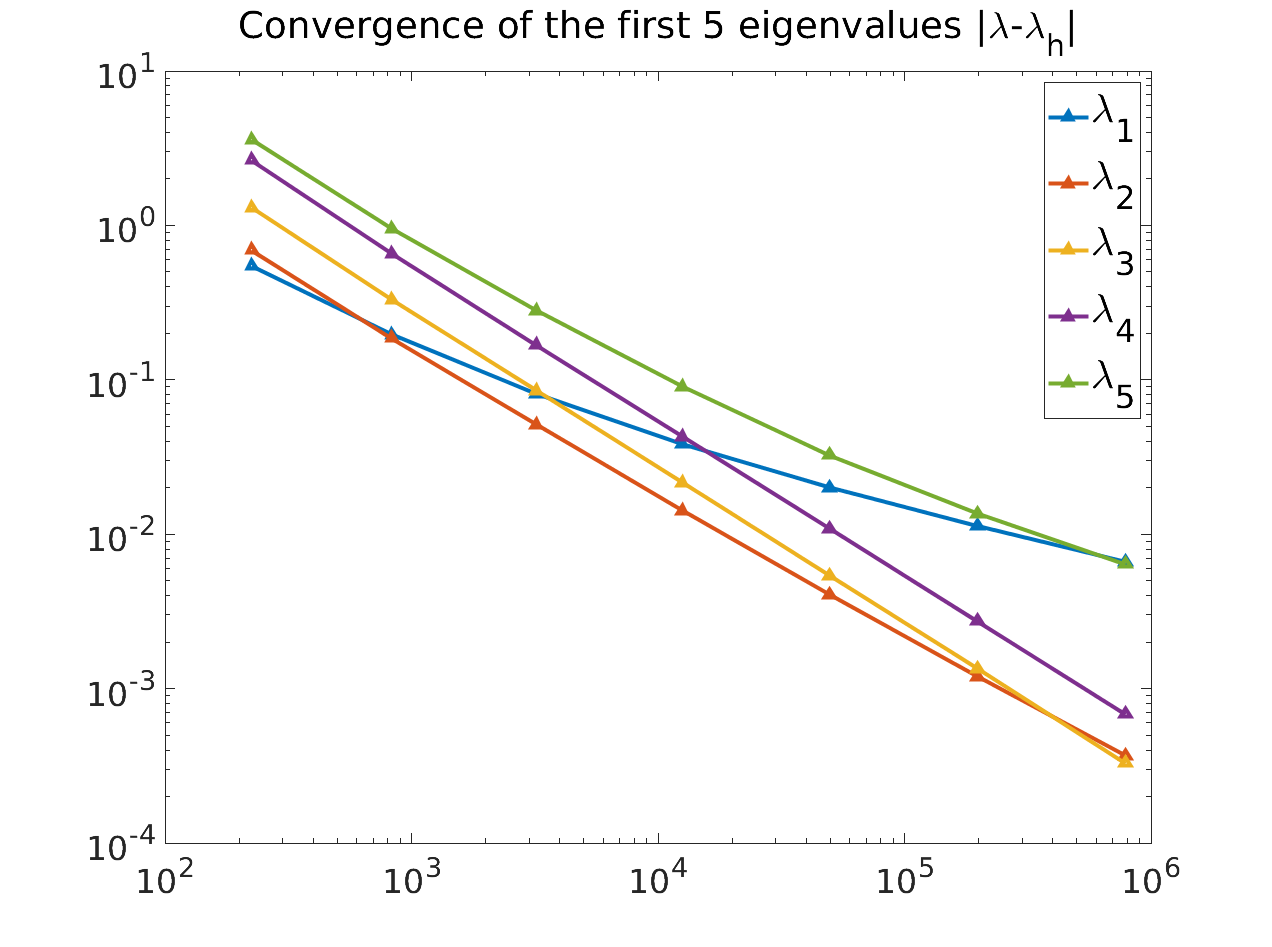}
\caption{Error of the first five eigenvalues computed with the div-curl
formulation on an L-shaped domain}
\label{fg:Lshaped}
\end{figure}

\subsection{A posteriori analysis and adaptive algorithm}

The a posteriori error estimator studied in Section~\ref{se:post} can be
naturally used in order to drive an adaptive scheme within the usual
SOLVE--ESTIMATE--MARK--REFINE cycle, when D\"orfler marking is adopted.
We used the FOSLS formulation with Raviart--Thomas elements in order to
approximate the fundamental mode of the Laplace eigenvalue problems on an
L-shaped domain. Figure~\ref{fg:adaptive} shows the error plots as a function
of the number of degrees of freedom corresponding to different choices of the
D\"orfler bulk parameter $\vartheta$. Uniform refinement corresponds to the
choice $\vartheta=1$. The results show that the choice $\vartheta=0.3$ gives
optimal convergence.
\begin{figure}
\includegraphics[width=9cm]{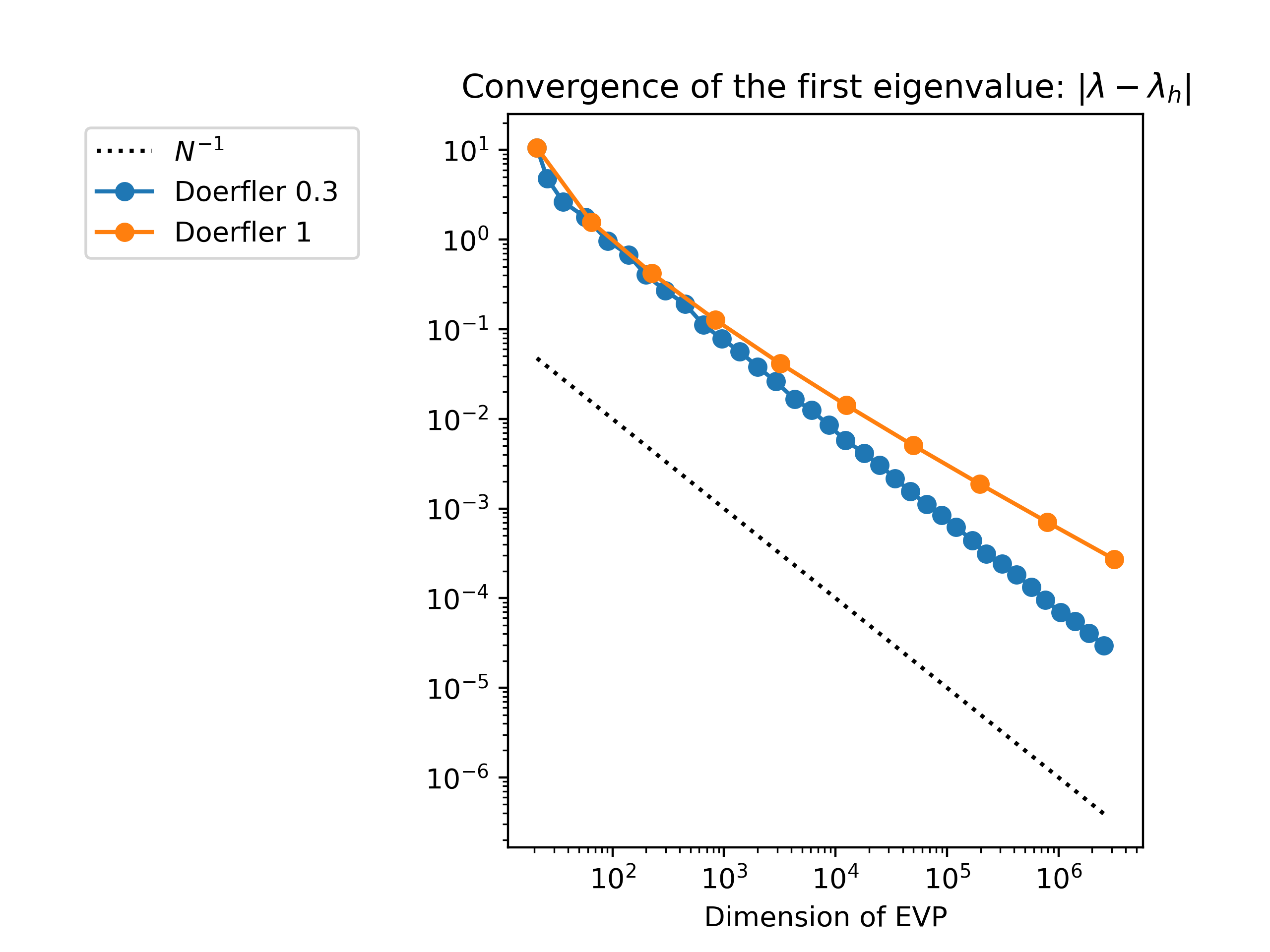}
\caption{Adaptive scheme: convergence of the first eigenvalue depending on
different choices of the D\"orfler bulk parameter}
\label{fg:adaptive}
\end{figure}

\bibliographystyle{amsplain}
\bibliography{ref}

\end{document}